\documentclass[12pt, notitlepage]{amsart}
\usepackage{latexsym, amsfonts, amsmath, amssymb, amsthm}

\newcommand{\tmscript}[1]{\text{\scriptsize{$#1$}}}
\newtheorem{lemma}{Lemma}

\newtheorem{theorem}{Theorem}

\newtheorem{definition}{Definition}

\newcommand{\assign}{:=}

\newcommand{\mathd}{\mathrm{d}}
\newcommand{\mathi}{\mathrm{i}}
\newcommand{\tmop}{\text}
\newcommand{\um}{-}
\newcommand{\upl}{+}
\newcommand{\tmem}{\textit}

\newcommand{\tmtextbf}{\textbf}

\usepackage{graphicx}
\usepackage{mathrsfs}

\begin{document}

\title{A converse to Hal\'asz's theorem}

\begin{abstract}
  We show that the distribution of large values of an additive function on the integers, and the distribution
  of values of the additive function on the primes are related to each other via a Levy Process. 
  As a consequence we obtain a converse to an old theorem of Halasz \cite{Halasz}. Halasz proved that if
  $f$ is an strongly additive function with $f(p) \in \{0,1\}$, then $f$ is Poisson distributed
  on the integers. We prove, conversely, that if $f$ is Poisson distributed on the integers then for most primes $p$,
  $f(p) = o(1)$ or $f(p) = 1 + o(1)$. 
\end{abstract}

\author{Maksym Radziwi\l\l}
\address{Department of Mathematics \\ Stanford University\\
450 Serra Mall, Bldg. 380\\Stanford, CA 94305-2125}
\email{maksym@stanford.edu}
\thanks{The author is partially supported by a NSERC PGS-D award}

\maketitle

\section{Introduction.}

Let $g$ be an strongly additive function (that is, $g(m n) = g(m) + g(n)$ for $(m,n) = 1$ and $g(p^k) = g(p)$). According to a model due to Mark
Kac \cite{Kac} the distribution of the $g (n)$'s (with $n \leqslant x$ and $x$ large) is
predicted by the random variable $\sum_{p \leqslant x} g (p) X_p$ where the
$X_p$'s are independent Bernoulli random variables with
\[ \mathbb{P}(X_p = 1) = 1 -\mathbb{P}(X_p = 0) = \frac{1}{p} \]
According to this model most $g (n)$'s cluster around the mean $\mu (g ; x)$
of $\sum_{p \leqslant x} g (p) X_p$ and within a constant multiple of $B (g
; x)$. Here,
\[ \mu (g ; x) \assign \sum_{p \leqslant x} \frac{g (p)}{p} \text{ \ and \ }
   B (g ; x)^2 \assign \sum_{p \leqslant x} \frac{g (p)^2}{p} \]
In this paper (a follow up to \cite{Maks}) we investigate the inter-relation between
the distribution of an additive function $g$ on the integers, i.e,
\begin{equation}
  \label{integers}
  \frac{1}{x} \cdot \# \left\{ n \leqslant x : \frac{g (n) - \mu (g ; x)}{B (g
  ; x)} \geqslant \Delta \right\}
\end{equation}
in the range $1 \leqslant \Delta \ll B (g ; x)^{1-\varepsilon}$, and the
distribution of the additive function $g$ on the primes:
\begin{equation}
  \label{primes}
  \lim_{x \rightarrow \infty} \frac{1}{B (g ; x)^{2}}
  \sum_{\tmscript{\begin{array}{c}
    p \leqslant x\\
    g (p) \leqslant t
  \end{array}}} \frac{g (p)^2}{p}
\end{equation}
for fixed $t$. Previously this question was considered only in the context
of limit theorems (that is, with $\Delta = O (1)$ in $(1)$) and in that
situation the behavior of (\ref{integers}) 
is controlled by the distribution of $g (p) / B (g ; x)$ for $p \leqslant x$ (see \cite{Elliott}, p. 12). 

We will prove, roughly speaking, that in the range
$1 \leq \Delta \leq B(g;x)^{1 - \varepsilon}$, (\ref{integers}) behaves as 
a sum of $B (f ; x)^{2}$
independent and identically distributed random variables if and only if 
(\ref{primes})
converges to a distribution function for almost all $t$. Since $B (f ; x)^{2}$
might not be an integer there is some care necessary in defining what it means
to have ``$B (f ; x)^{2}$ random variables''. For this reason we first
illustrate our results in the special case of the Poisson distribution.

Following the work of Selberg \cite{Selberg} and Sathe \cite{Sathe} (in the case $g (p) = 1$) and Halasz
\cite{Halasz} (in the general case $g (p) \in \{0, 1\}$), it is known that if $g$ is a
strongly additive function with $g (p) \in \{0, 1\}$ and $B (g ; x)
\rightarrow \infty$, then
\begin{equation}
  \label{Poisson}
  \frac{1}{x} \cdot \# \left\{ n \leqslant x : \frac{g (n) - \mu (g ; x)}{B (g
  ; x)} \geqslant \Delta \right\} \sim \mathbb{P} \tmop{oisson} \left( B (g
  ; x)^{2} ; \Delta \right)
\end{equation}
uniformly in $1 \leqslant \Delta \leqslant o (B (g ; x))$, where,
\[ \mathbb{P} \tmop{oisson} \left( \lambda, \Delta \right) \assign \sum_{v
   \geqslant \lambda + \sqrt[]{\lambda} \Delta} e^{- \lambda} \cdot
   \frac{\lambda^v}{v!} \]
denotes the tail of a Poisson distribution with parameter $\lambda$. It is not
difficult to see that the condition $g (p) \in \{0, 1\}$ is not strictly
necessary for the validity of (\ref{Poisson}) (for example one could let $g (2) = 10$
and $(3)$ would still hold). Indeed, 
a minor variation on Halasz's theorem \cite{Halasz}
shows that to guarantee (\ref{Poisson}) it is enough to have $0 \leqslant g (p)
\leqslant O (1)$, $B (g ; x) \rightarrow \infty$ and
\begin{equation}
  \label{primecondition}
  \frac{1}{B (g ; x)^{2}} \sum_{\tmscript{\begin{array}{c}
    p \leqslant x\\
    g (p) \leqslant t
  \end{array}}} \frac{g (p)^2}{p} = \delta_1 (t) + O \left( \frac{1}{B (g ;
  x)^{2}} \right)
\end{equation}
as $x \rightarrow \infty$, uniformly in $t \in \mathbb{R}$, and where
\begin{equation}
\label{deltaone}
\delta_{\alpha} (t) \assign \begin{cases}
0 & \text{ if } t < \alpha \\
\tfrac 12 & \text{ if } t = \alpha \\
1 & \text{ if } t > \alpha
\end{cases}
\end{equation}
Relation (\ref{primecondition}) roughly asserts that for most primes $p$, $g (p) = o (1)$ or $g
(p) = 1 + o (1)$. It turns out that a condition such as (\ref{primecondition}) (perhaps with a
weaker error term) is not only sufficient for (\ref{Poisson}) but also necessary. This
is a special case of our main result (Theorem \ref{MainThm}).

\begin{theorem}
  \label{PoissonTheorem}
  Let $g$ be a strongly additive function. If $(3)$ holds uniformly in $1
  \leqslant \Delta \ll B (g ; x)^{1 - \varepsilon}$, then, as $x \rightarrow
  \infty$,
  \[ \frac{1}{B (g ; x)^{2}} \sum_{\tmscript{\begin{array}{c}
       p \leqslant x\\
       g (p) \leqslant t
     \end{array}}} \frac{g (p)^2}{p} \longrightarrow \delta_1 (t). \]
  for all $t \neq 1$. 
\end{theorem}

Both Theorem \ref{PoissonTheorem} and the sufficient condition (\ref{primecondition}) are special cases of
more general results which we will now state. Our proofs build in a fundamental way on Halasz's paper \cite{Halasz}
and its application to large deviations of additive functions developed 
by Halasz (unpublished, see \cite{Elliott}) 
and the Lithuanian school (for
example \cite{Maciulis}). The main tool underneath these large deviations 
results is Cramer's method \cite{Cramer}, 
first transposed into an arithmetic context by Kubilius, in \cite{Kubilius}. 

Despite being more technical our main result (Theorem \ref{MainThm} and Theorem \ref{MainThm2}) has
the merit of making clear the inter-relation between the distribution of the
additive function on the integers and its distribution on the primes. This
inter-relation  by Levy Processes (in fact a modified version of
that concept) which we now introduce into our discussion.

\begin{definition}
  Let $\Psi$ be a distribution function. Suppose that $\Psi (\alpha) - \Psi
  (0) = 1$ for some $\alpha > 0$. We define $\{\mathcal{Z}_{\Psi} (u)\}_{u
  \geqslant 0}$ to be a sequence of random variables with distribution
  determined by
  \[ \mathbb{E} \left[ e^{\mathi t\mathcal{Z}_{\Psi} (u)} \right] = \exp
     \left( \mathi u \cdot \int_0^{\infty} \frac{e^{\mathi tx} - \mathi tx -
     1}{x^2} \cdot \mathd \Psi (x) \right) \]
  for each $u \geqslant 0$.
\end{definition}

For integer $u \in \mathbb{N}$, the random variable $\mathcal{Z}_{\Psi} (u)$
can be thought off as a sum of $u$ independent copies of $\mathcal{Z}_{\Psi}
(1)$. Thus $\mathcal{Z}_{\Psi} (u)$ is a continuous generalization of the
notion of a random walk. Notice also that if $\Psi (t) = \delta (t)$, then
$\mathcal{Z}_{\Psi} (u)$ is a centered Poisson random variable with parameter
$u$. \

Theorem \ref{PoissonTheorem} is a particular case (i.e $\Psi (t) = \delta_1 (t))$ of the
following result. 

\begin{theorem}
  \label{MainThm}
  Let $g$ be a strongly additive function with $0 \leqslant g (p) \leqslant O
  (1)$ and $B (g ; x) \rightarrow \infty$. If there is a distribution function
  $\Psi (t) \neq \delta_0 (t)$ (such that $\Psi (\alpha) - \Psi (0) = 1$ for
  some $\alpha > 0$) and such that,
  \begin{equation}
    \label{assumption}
    \frac{1}{x} \cdot \# \left\{ n \leqslant x : \frac{g (n) - \mu (g ; x)}{B
    (g ; x)} \geqslant \Delta \right\} \sim \mathbb{P} \left(
    \mathcal{Z}_{\Psi} (B (g ; x)^{2}) \geqslant \Delta B(g;x) \right)
  \end{equation}
  uniformly in $1 \leqslant \Delta \ll B (g ; x)^{1 - \varepsilon}$, then,
  \begin{equation}
    \label{conclusion}
    \lim_{x \rightarrow \infty} \frac{1}{B (g ; x)^{2}}
     \sum_{\tmscript{\begin{array}{c}
       p \leqslant x\\
       g (p) \leqslant t
     \end{array}}} \frac{g (p)^2}{p} = \Psi (t)
     \end{equation}
  at all continuity points of $\Psi (t)$.
\end{theorem}
The range $1 \leq \Delta \ll B(g;x)^{1 - \varepsilon}$ is optimal for the conclusion of the Theorem. In the range $1 \leq \Delta \ll B(g;x)^{\alpha}$ (with
$\alpha < 1$) the left-hand side of (\ref{assumption}) depends only on
the first $k \leq \lceil \frac{1 + \alpha}{1 - \alpha} \rceil$ moments 
$$
\frac{1}{B(g;x)^{2}}\sum_{p \leq x} \frac{g(p)^{2 + k}}{p}
$$
and thus cannot imply (\ref{conclusion}) (see
\cite{Maks} for more on this).
It is fitting to complete Theorem \ref{MainThm} with a result showing that each of the
laws $\mathcal{Z}_{\Psi} (u)$ does occur as the law of {\tmem{some}} additive
function (and thus (\ref{assumption}) is not vacuous for any $\Psi$). We show this below. \

\begin{theorem}
  \label{MainThm2}
  Let $g$ be a strongly additive function with $0 \leqslant g (p) \leqslant O
  (1)$ and $B (g ; x) \rightarrow \infty$. If there is a distribution function
  $\Psi (t) \neq \delta_0 (t)$ with $\Psi (\alpha) - \Psi (0) = 1$ (for some
  $\alpha > 0$) such that
  \begin{equation}
    \label{assumption2}
    \frac{1}{B (g ; x)^{2}} \sum_{\tmscript{\begin{array}{c}
      p \leqslant x\\
      g (p) \leqslant t
    \end{array}}} \frac{g (p)^2}{p} = \Psi (g ; t) + 
    O \left( \frac{1}{B (g ;
    x)^{2}} \right)
  \end{equation}
  uniformly in $t$, then,
  \[ \frac{1}{x} \cdot \# \left\{ n \leqslant x : \frac{g (n) - \mu (g ; x)}{B
     (g ; x)} \geqslant \Delta \right\} \sim \mathbb{P} \left(
     \mathcal{Z}_{\Psi} (B (g ; x)^{2}) \geqslant \Delta B(g;x) \right) \]
  uniformly in $1 \leqslant \Delta \leqslant o (B (g ; x))$.
\end{theorem}

For each $\Psi (t)$ one can always choose the $g (p)$'s so that (\ref{assumption2}) is
satisfied. It follows that for each $\mathcal{Z}_{\Psi}$ there is an additive
function $g$ whose distribution on the integers is described by the process
$\{\mathcal{Z}_{\Psi} (u)\}_{u \geqslant 0}$. In particular the assumption
(\ref{assumption}) of Theorem \ref{MainThm} is never vacuous.

The conclusion of these two theorems can be subsumed as follows: given an
additive function $g$, its distribution on the primes $\Psi$ and its
distribution on the integers $F$, a related through the following
correspondance,
\[ \left\{ \mathcal{Z}_{\Psi} (u)\}= F \longleftrightarrow \Psi (t) \right. \]
It would be interesting to elucidate this connection further, and prove
theorems that allow for a $\Psi$ that varies with $x$, say,
\[ \Psi_x (t) \assign \frac{1}{B (g ; x)^{2}} \sum_{\tmscript{\begin{array}{c}
     p \leqslant x\\
     g (p) \leqslant t
   \end{array}}} \frac{g (p)^2}{p} \]
Then the goal would be to describe the resulting distribution
$\mathcal{Z}_{\Psi_x} (B (g ; x)^{2})$ on the integers in terms of $\Psi_x (t)$,
and vice-versa. The questions appearing in this paper have been also explored
in a different context in my paper \cite{Maks}

\section{Integers to primes : Proof of Theorem \ref{MainThm}}

The goal of this section is to prove Theorem \ref{MainThm}. Throughout we define
$\mathcal{D}_f (x ; \Delta)$ by
\[ \mathcal{D}_f (x ; \Delta) \assign \frac{1}{x} \cdot \# \left\{ n
   \leqslant x : \frac{f (n) - \mu (f ; x)}{B (f ; x)} \geqslant \Delta
   \right\} \]

\subsection{Large deviations for $\mathcal{D}_f (x ; \Delta)$ and
$\mathbb{P}(\mathcal{Z}_{\Psi} (x) \geqslant t)$}

We will usually need to ``adjust'' some of the results taken from the
literature. Our main tool will be Lagrange inversion.

\begin{lemma}
  Let $C > 0$ be given. Let $f (z)$ be analytic in $|z| \leqslant C$. Suppose
  that $f' (z) \neq 0$ for all $|z| \leqslant C$ and that in $|z| \leqslant C$
  the function $f (z)$ vanishes only at the point $z = 0$. Then, the function
  $g$ defined implicitly by $f (g (z)) = z$ is analytic in a neighborhood of
  0 and its n-th coefficient $a_n$ in the Taylor expansion
  about 0 is given by
  \[ a_n = \frac{1}{2 \pi i} \oint_{\gamma} \frac{\zeta f' (\zeta)}{f
     (\zeta)^{n + 1}} \mathd \zeta \]
  where $\gamma$ is a circle about $0$, contained in $| \zeta | \leqslant C$.
  The function $g (z)$ is given by
  \begin{eqnarray*}
    g (z) & = & \frac{1}{2 \pi i} \oint_{\gamma} \frac{\zeta f' (\zeta)}{f
    (\zeta) - z} \cdot \mathd \zeta
  \end{eqnarray*}
  and again $\gamma$ is a circle about $0$, contained in $| \zeta | \leqslant
  C$. 
\end{lemma}

The desired asymptotic for $\mathcal{D}_f (x ; \Delta)$ is contained
in Maciulis's paper (\cite{Maciulis}, lemma 1A). 

\begin{lemma}
  \label{MainLemma1}
  Let $f$ be an additive function. Suppose that $0 \leqslant f (p) \leqslant O
  (1)$ and that $B (f ; x) \longrightarrow \infty$ (or equivalently $\sigma (f
  ; x) \rightarrow \infty$). Let $B^{2} = B (f ; x)^{2}$. Uniformly in the range
  $1 \leqslant \Delta \leqslant o (\sigma (f ; x))$ we have
  \[ \mathcal{D}_f (x ; \Delta) \sim \exp \left( - \frac{\Delta^3}{B}
     \sum_{k = 0}^{\infty} \frac{\lambda_f (x ; k + 2)}{k + 3} \cdot \left(
     \Delta / B \right)^k \right) \int_{\Delta}^{\infty} e^{- u^2 / 2} \cdot
     \frac{\mathd u}{\sqrt[]{2 \pi}} \]
  where the coefficients $\lambda_f (x ; k)$ are defined recursively by
  $\lambda_f (x ; 0) = 0, \lambda_f (x ; 1) = 1$ and
  \[ \lambda_f (x ; j) = - \sum_{i = 2}^j \frac{1}{i!} \cdot \left(
     \frac{1}{B (f ; x)^{2}} \sum_{p \leqslant x} \frac{f (p)^{i + 1}}{p}
     \right) \sum_{k_1 + \ldots + k_i = j} \lambda_f (x ; k_1) \cdot \ldots
     \cdot \lambda_f (x ; k_i) \]
  Further there is a constant $C = C (f)$ such that $| \lambda_f (x ; k) |
  \leqslant C^k$ for all $k, x \geqslant 1$.
\end{lemma}

\begin{proof}
  Except the bound $| \lambda_f (x ; k) | \leqslant C^k$, the totality of the
  lemma is contained in Maciulis's paper (\cite{Maciulis}, lemma 1A). Let us prove
  that $| \lambda_f
  (x ; k) | \leqslant C^k$ for a suitable positive constant $C > 0$. To do so,
  we consider the power series
  \begin{eqnarray*}
    \mathcal{G}_f \left( x ; z \right) & = & \sum_{j \geqslant 2} \lambda_f
    \left( x ; j \right) \cdot z^j + z
  \end{eqnarray*}
  Let us look in more detail at the sum over $j \geqslant 2$. By making use of
  the recurrence relation for $\lambda_f (x ; j)$ we see that the sum in
  question equals to
  \begin{eqnarray*}
    & = & - \sum_{j \geqslant 2} \sum_{i = 2}^j \frac{1}{i!} \cdot \left(
    \frac{1}{B \left( f ; x \right)^{2}} \sum_{p \leqslant x} \frac{f (p)^{i +
    1}}{p} \right) \sum_{k_1 + \ldots + k_i = j} \lambda_f \left( x ; k_1
    \right) \cdot \ldots \cdot \lambda_f \left( x ; k_i \right) \cdot z^j\\
    & = & - \sum_{i \geqslant 2} \frac{1}{i!} \cdot \left( \frac{1}{B
    \left( f ; x \right)^{2}} \sum_{p \leqslant x} \frac{f (p)^{i + 1}}{p} \right)
    \cdot \left( \sum_{k \geqslant 0} \lambda_f \left( x ; k \right) z^k
    \right)^i\\
    & = & - \frac{1}{B (f ; x)^{2}} \sum_{p \leqslant x} \frac{f (p)}{p}
    \sum_{i \geqslant 2} \frac{1}{i!} \cdot f (p)^i \cdot \mathcal{G}_f \left(
    x ; z \right)^i\\
    & = & - \frac{1}{B \left( f ; x \right)^{2}} \sum_{p \leqslant x} \frac{f
    (p)}{p} \cdot \left( e^{f (p)\mathcal{G}_f \left( x ; z) \right.} - f
    (p)\mathcal{G}_f (x ; z) - 1 \right)_{}
  \end{eqnarray*}
  The above calculation reveals that $\mathcal{F}_f (x ; \mathcal{G}_f (x ;
  z)) = z$ where $\mathcal{F}_f (x ; z)$ is defined by
  \begin{eqnarray*}
    \mathcal{F}_f (x ; z) & = & z + \frac{1}{B (f ; x)^{2}} \sum_{p \leqslant x}
    \frac{f (p)}{p} \cdot \left( e^{f (p) z} - f (p) z - 1 \right)\\
    & = & \frac{1}{B (f ; x)^{2}} \sum_{p \leqslant x} \frac{f (p)}{p} \cdot
    \left( e^{f (p) z} - 1 \right)
  \end{eqnarray*}
  Since all the $f (p)$ are bounded by some $M \geqslant 0$, we have
  $\mathcal{F}_f (x ; z) = z + O \left( Mz^2 \right)$ when $z$ is in a
  neighborhood of $0$, furthermore the implicit constant in the big $O$,
  depends only on $M$. Therefore $\mathcal{F}_f (x ; z) \gg 1$ for $z$ in the
  annulus $B / 2 \leqslant |z| \leqslant B$, where $B$ is a sufficiently small
  constant, depending only on $M$. Let us also note the derivative
  \[ \frac{\mathd}{\mathd z} \cdot \mathcal{F}_f (x ; z) = \frac{1}{B (f ;
     x)^{2}} \sum_{p \leqslant x} \frac{f (p)^2}{p} \cdot e^{f (p) z} \]
  doesn't vanish and is bounded uniformly in $|z| \leqslant B$ for $B$
  sufficiently small, depending only on $M$. Hence by Lagrange inversion the
  function $\mathcal{G}_f (x ; z)$ is for each $x \geqslant 1$ analytic in the
  neighborhood $|z| \leqslant B$ of $0$, and in addition, its coefficients
  $\lambda_f (x ; k)$ are given by
  \begin{eqnarray*}
    \lambda_f (x ; k) & = & \frac{1}{2 \pi i} \oint_{| \zeta | = B / 2}
    \frac{\zeta \cdot (\mathd / \mathd \zeta)\mathcal{F}_f (x ;
    \zeta)}{\mathcal{F}_f (x ; \zeta)^{k + 1}} \mathd \zeta
  \end{eqnarray*}
  However we know that $\mathcal{F}_f (x ; \zeta) \gg 1$ and that $(d / d
  \zeta)\mathcal{G}_f (x ; \zeta) \ll 1$ on the boundary $| \zeta | = B / 2$
  with the implicit constant depending only on $M$. Therefore, the integral is
  bounded by $C^k$, for some $C > 0$ depending only on $M$. Hence $| \lambda_f
  (x ; k) | \leqslant C^k$.
\end{proof}

From Hwang's paper \cite{Hwang} -- itself heavily based on the same methods as
used by Maciulis \cite{Maciulis} -- we obtain the next lemma. 
Since Hwang's lemma
is not exactly what is stated below, we include the deduction.

\begin{lemma}
  \label{MainLemma2}
  Let $\Psi$ be a distribution function. Suppose that there is an $\alpha > 0$
  such that $\Psi (\alpha) - \Psi (0) = 1$. Let $B^2 = B^2 (x) \rightarrow
  \infty$ be some function tending to infinity. Uniformly in the range $1
  \leqslant \Delta \leqslant o (B^{})$ we have
  \[ \mathbb{P} \left( \mathcal{Z}_{\Psi} \left( B^2 \right) \geqslant \Delta
     B \right) \sim \exp \left( - \frac{\Delta^3}{B} \sum_{k = 0}^{\infty}
     \frac{\Lambda (\Psi ; k + 2)}{k + 3} \cdot \left( \Delta / B \right)^k
     \right) \int_{\Delta}^{\infty} e^{- u^2 / 2} \cdot \frac{\mathd
     u}{\sqrt[]{2 \pi}} \]
  the coefficients $\Lambda (\Psi ; k)$ satisfy $\Lambda (\Psi ; 0) = 0,
  \Lambda (\Psi ; 1) = 1$ and the recurrence relation
  \[ \Lambda (\Psi ; j) = - \sum_{2 \leqslant \ell \leqslant j} \frac{1}{\ell
     !} \int_{\mathbb{R}} t^{\ell - 1} \mathd \Psi (f ; t) \sum_{k_1 + \ldots
     + k_{\ell} = j} \Lambda (\Psi ; k_1) \cdot \ldots \cdot \Lambda (\Psi ;
     k_{\ell}) \]
  Furthermore there is a constant $C = C (\Psi) > 0$ such that $| \Lambda
  (\Psi ; k) | \leqslant C^k$ for $k \geqslant 1$.
\end{lemma}

\begin{proof}
  Let $u (z) = u (\Psi ; z) = \int_{\mathbb{R}} (e^{zt} - zt - 1) \cdot t^{-
  2} \mathd \Psi (t)$. Note that $u (z)$ is entire because $\Psi (t)$ is
  supported on a compact interval. By Hwang's theorem 1 (see \cite{Hwang})
  \[ \mathbb{P} \left( \mathcal{Z}_{\Psi} \left( B^2 \right) \geqslant \Delta
     B \right) \sim \exp \left( - B^2 \sum_{k \geqslant 0} \frac{\Lambda (\Psi
     ; k + 2)}{k + 3} \cdot (\Delta / B)^k \right) \int_{\Delta}^{\infty} e^{-
     u^2 / 2} \cdot \frac{\mathd u}{\sqrt[]{2 \pi}} \]
  uniformly in $1 \leqslant \Delta \leqslant o (B (x))$ with the coefficients
  $\Lambda (\Psi ; k)$ given by $\Lambda (\Psi ; 0) = 0$, $\Lambda (\Psi ; 1)
  = 1$ and for $k \geqslant 0$,
  \begin{eqnarray*}
    \frac{\Lambda (\Psi ; k + 2)}{k + 3} & = & - \frac{1}{k + 3} \cdot
    \frac{1}{2 \pi i} \oint_{\gamma} u'' (z) \cdot \left( \frac{u' (z)}{z}
    \right)^{- k - 3} \cdot \frac{\mathd z}{z^{k + 2}}\\
    & = & - \frac{1}{k + 3} \oint_{\gamma} \frac{zu'' (z)}{u' (z)^{k + 3}}
    \cdot \frac{\mathd z}{2 \pi i}
  \end{eqnarray*}
  
  %% Come back here

  (we set $m = k+3$, $k \geqslant 0$, $q_m = \Lambda(\Psi;k+2)/(k+3)$ in
  equation $(7)$ of \cite{Hwang} and rewrite equation $(8)$ in \cite{Hwang}
  in terms of Cauchy's formula).  
  Here $\gamma$ is a small circle around the origin. First let us show that
  the coefficients $\Lambda (\Psi ; k + 2)$ are bounded by $C^k$ for a
  sufficiently large (but fixed) $C > 0$. Around $z = 0$ we have $u' (z) =
  \int_{\mathbb{R}} (e^{zt} - 1) \cdot t^{- 1} \mathd \Psi (t) = z + O
  (z^2)$. Therefore if we choose the circle $\gamma$ to have sufficiently
  small radius then $u' (z) \gg 1$ for $z$ on $\gamma$. Hence looking at the
  previous equation, the Cauchy integral defining $\Lambda (\Psi ; k + 2) /
  k + 3$ is bounded in modulus by $\ll C^{k + 3}$ for some constant $C > 0$.
  The bound $| \Lambda (\Psi ; k) | \ll C^k$ ensues (perhaps with a larger $C$
  than earlier). Our goal now is to show that $\Lambda (\Psi ; k)$ satisfies
  the recurrence relation given in the statement of the lemma. Multiplying by
  $\xi^{k + 3}$ and summing over $k \geqslant 0$ we obtain
  \begin{eqnarray*}
    &  & \sum_{k \geqslant 0} \frac{\Lambda (\Psi ; k + 2)}{k + 3} \cdot
    \xi^{k + 3} \text{ } = \text{ } - \sum_{k \geqslant 0} \frac{\xi^{k +
    3}}{k + 3} \oint_{\gamma} \frac{zu'' (z)}{u' (z)^{k + 3}} \cdot
    \frac{\mathd z}{2 \pi i}\\
    & = & - \oint_{\gamma} zu'' (z) \sum_{k \geqslant 0} \frac{1}{k + 3}
    \cdot \left( \frac{\xi}{u' (z)} \right)^{k + 3} \cdot \frac{\mathd z}{2
    \pi i}\\
    & = & \oint_{\gamma} zu'' (z) \cdot \left( - \frac{\xi}{u' (z)} -
    \frac{1}{2} \cdot \frac{\xi^2}{u' (z)^2} - \log \left( 1 - \frac{\xi}{u'
    (z)} \right) \right) \cdot \frac{\mathd z}{2 \pi i}
  \end{eqnarray*}
  Differentiating with respect to $\xi$ on both sides yields
  \begin{eqnarray*}
    \mathcal{G}_{\Psi} (\xi) & = & \sum_{k \geqslant 0} \Lambda (\Psi ; k + 2)
    \xi^{k + 2} \text{ } = \text{ } \oint \left( \frac{zu'' (z)}{u' (z) - \xi}
    - \frac{zu'' (z)}{u' (z)} - \frac{\xi zu'' (z)}{u' (z)^2} \right) 
    \frac{\mathd z}{2 \pi i}
  \end{eqnarray*}
  By Lagrange inversion this last integral is equal to $(u')^{- 1} (\xi) - 0 -
  \xi$ where $(u')^{- 1}$ denotes the inverse function to $u' (z)$. Hence
  \begin{eqnarray}
    \label{uinverse}
    \sum_{k \geqslant 0} \Lambda (\Psi ; k) \xi^k & = & (u')^{- 1} (\xi) 
  \end{eqnarray}
  Let us compose this with $u' (\cdot)$ on both sides and compute the
  resulting left hand side. First of all we expand $u' (z)$ in a power series.
  This gives
  \[ u' (z) = \int_{\mathbb{R}} \frac{e^{zt} - 1}{t} \mathd \Psi (t) =
     \sum_{\ell \geqslant 1} \frac{1}{\ell !} \int_{\mathbb{R}} t^{\ell - 1}
     \mathd \Psi (t) \cdot z^{\ell} \]
  Therefore, composing (\ref{uinverse}) with $u' (\cdot)$ yields
  \begin{eqnarray*}
    \xi & = & u' \left( \sum_{k \geqslant 0} \Lambda (\Psi ; k) \xi^k \right)
    \text{ } = \text{ } \sum_{\ell \geqslant 1} \frac{1}{\ell !}
    \int_{\mathbb{R}} t^{\ell - 1} \mathd \Psi (t) \cdot \left( \sum_{k
    \geqslant 0} \Lambda (\Psi ; k) \xi^k \right)^{\ell}\\
    & = & \sum_{\ell \geqslant 1} \frac{1}{\ell !} \int_{\mathbb{R}} t^{\ell
    - 1} \mathd \Psi (t) \cdot \left( \sum_{k_1, \ldots, k_{\ell} \geqslant 1}
    \Lambda (\Psi ; k_1) \cdot \ldots \cdot \Lambda (\Psi ; k_{\ell}) \cdot
    \xi^{k_1 + \ldots + k_{\ell}} \right)\\
    & = & \sum_{m \geqslant 1} \left( \sum_{1 \leqslant \ell \leqslant m}
    \frac{1}{\ell !} \int_{\mathbb{R}} t^{\ell - 1} \mathd \Psi (t) \sum_{k_1
    + \ldots + k_{\ell} = m} \Lambda (\Psi ; k_1) \cdot \ldots \cdot \Lambda
    (\Psi ; k_{\ell}) \right) \cdot \xi^m
  \end{eqnarray*}
  Thus the first coefficient $\Lambda (\Psi ; 1)$ is equal to 1, as desired,
  while for the terms $m \geqslant 2$ we have
  \[ \sum_{1 \leqslant \ell \leqslant m} \frac{1}{\ell !} \int_{\mathbb{R}}
     t^{\ell - 1} \mathd \Psi (t) \sum_{k_1 + \ldots + k_{\ell} = m} \Lambda
     (\Psi ; k_1) \cdot \ldots \cdot \Lambda (\Psi ; k_{\ell}) = 0 \]
  The first term $\ell = 1$ is equal to $\Lambda (\Psi ; m)$. It suffice to
  move it on the right hand side of the equation, to obtain the desired
  recurrence relation. 
\end{proof}

Finally we will need one last result ``from the literature''. Namely a weak
form of the method of moments. For a proof we refer the reader to Gut's book
\cite{Gut}, p. 237. (Note that the next lemma follows from the result in
\cite{Gut} because in our case the random variables are positive, and bounded, 
in particular their distribution is determined uniquely by their moments).

\begin{lemma}
  \label{momentlemma}
  Let $\Psi$ be a distribution function. Suppose that there is an $a > 0$ such
  that $\Psi (a) - \Psi (0) = 1$. Let $F (x ; t)$ be a sequence of
  distribution functions, one for each $x > 0$. If for each $k \geqslant 0$,
  \[ \int_{\mathbb{R}} t^k \tmop{dF} (x ; t) \text{ } \longrightarrow
     \text{\, } \int_{\mathbb{R}} t^k d \Psi (t) \]
  Then $F (x ; t) \longrightarrow \Psi (t)$ at all continuity points $t$ of
  $\Psi (t)$. 
\end{lemma}

\subsection{Proof of Theorem \ref{MainThm}}

\begin{proof}[Proof of Theorem \ref{MainThm}]
  By assumptions $\mathcal{D}_f (x ; \Delta) \sim
  \mathbb{P}(\mathcal{Z}_{\Psi} (B^2 (f ; x)) \geqslant \Delta B (f
  ; x))$ holds throughout $1 \leqslant \Delta \leqslant o (B (f ; x))$.

  The proof is in three steps. Retaining the notation of Lemma \ref{MainLemma1} and Lemma
  \ref{MainLemma2} we first show that $\lambda_f (k ; x) \longrightarrow \Lambda (\Psi ;
  k)$ for all $k \geqslant 2$ (for $k = 1$ this is trivial). Then, we deduce
  from there that
  \begin{equation}
    \frac{1}{B(f ; x)^{2}} \sum_{\tmscript{\begin{array}{c}
      p \leqslant x\\
      f (p) \leqslant t
    \end{array}}} \frac{f (p)^{k + 2}}{p} \text{ } \longrightarrow \text{ }
    \int_{\mathbb{R}} t^k \mathd \Psi (t)
  \end{equation}
  Hence by the method of moments
  \begin{eqnarray}
    \frac{1}{B (f ; x)^{2}} \sum_{\tmscript{\begin{array}{c}
      p \leqslant x\\
      f (p) \leqslant t
    \end{array}}} \frac{f (p)^2}{p} &
    \longrightarrow & \Psi (t) 
  \end{eqnarray}
  The last step being the easy one. To prove our first step we will proceed by
  induction on $k \geqslant 0$. We will prove the stronger claim that
  \[ \lambda_f (x ; k + 2) = \Lambda (\Psi ; k + 2) + O_k \left( B^{- 2^{- (k
     + 1)}} \right) \]
  where we write $B = B (f ; x)$ to simplify notation. By Lemma \ref{MainLemma1} and \ref{MainLemma2},
  our assumption $\mathcal{D}_f (x ; \Delta) \sim
  \mathbb{P}(\mathcal{Z}_{\Psi} (B (f ; x)^{2}) \geqslant \Delta B (f ; x))$
  (for $1 \leqslant \Delta \leqslant o (B (f ; x)))$ reduces to
  \begin{equation}
    \label{mainequality}
    - \frac{\Delta^3}{B} \sum_{m \geqslant 0} \frac{\lambda_f \left( x ; m +
    2) \right.}{m + 3} \cdot \left( \Delta / B \right)^m = -
    \frac{\Delta^3}{B} \sum_{m \geqslant 0} \frac{\Lambda (\Psi ; m + 2)}{m +
    3} \cdot \left( \Delta / B \right)^m + o (1)
  \end{equation}
  valid throughout the range $1 \leqslant \Delta \leqslant o (\sigma (f ;
  x))$. Let us first establish the base case $\lambda_f (x ; 2) = \Lambda
  (\Psi ; 2) + O (B^{- 1 / 2})$. In (\ref{mainequality}) we choose $\Delta = B^{1 / 2}$.
  Because of the bounds $| \lambda_f (x ; m) | \leqslant C^m$ and $| \Lambda
  (\Psi ; m) | \leqslant C^m$ (see lemma \ref{MainLemma1} and \ref{MainLemma2}) the terms 
  $m \geqslant 1$ contribute $O (1)$.
  The $m = 0$ term is $\asymp B^{1 / 2}$. It follows that $\lambda_f (x ;
  2) = \Lambda (\Psi ; 2) + O (B^{- 1 / 2})$ and so the base case follows. Let
  us now suppose that for all $\ell < k$, $(k \geqslant 1)$
  \[ \lambda_f \left( x ; \ell + 2 \right) = \Lambda \left( \Psi ; \ell + 2
     \right) + O_{\ell} \left( B^{- 2^{- (\ell + 1)}} \right) \]
  Note that we can assume (in the above equation) that the implicit constant
  depends on $k$, by taking the max of the implicit constants in $O_{\ell}(
  B^{-2^{-(\ell+1)}})$ for $\ell < k$. 
  In equation (\ref{mainequality}) let's choose $\Delta = B^{1 - 2^{-
  (k + 1)}}$. With this choice of $\Delta$ the terms that are $\geqslant k +
  1$ in (\ref{mainequality}) contribute at most
  \[ \left( \Delta^3 / B \right) \cdot \left( C \cdot \Delta / B \right)^{k +
     1} \ll_k B^2 \cdot B^{- 3 \cdot 2^{- (k + 1)}} \cdot B^{- (k + 1) \cdot
     2^{- (k + 1)}} \]
  on both sides of (\ref{mainequality}). On the other hand, we see (by using the induction
  hypothesis) that the terms $m \leqslant k - 1$ on the left and the right
  hand side of (\ref{mainequality}) differ by no more that
  \begin{eqnarray*}
    &  & - \frac{\Delta^3}{B} \cdot \left( \sum_{\ell < k} \frac{(\Delta /
    B)^{\ell}}{\ell + 3} \cdot \left( \lambda_f \left( x ; \ell + 2 \right) -
    \Lambda \left( \Psi ; \ell + 2 \right) \right) \right)\\
    & = & O_k \left( B^2 \cdot B^{- 3 \cdot 2^{- (k + 1)}} \cdot \left(
    \sum_{\ell < k} \frac{B^{- \ell \cdot 2^{- (k + 1)}}}{\ell + 3} \cdot B^{-
    2^{- (\ell + 1)}} \right) \right)\\
    & = & O_k \left( B^2 \cdot B^{- 3 \cdot 2^{- (k + 1)}} \cdot \sum_{\ell <
    k} \frac{1}{\ell + 3} \cdot B^{- 2^{- (k + 1)} \cdot \left( \ell + 2^{k -
    \ell} \right)} \right)
  \end{eqnarray*}
  Note that for each integer $\ell < k$ we have $\ell + 2^{k - \ell} \geqslant
  k + 1$. Therefore the above error term is bounded by $O_k (B^2 \cdot B^{- 3
  \cdot 2^{- k}} \cdot B^{- (k + 1) 2^{- (k + 1)}})$. With these two
  observations at hand, relation (\ref{mainequality}) reduces to
  \[ - \frac{\Delta^3}{B} \cdot \frac{\left( \Delta / B \right)^k}{k + 3}
     \left[ \lambda_f \left( x ; k \upl 2 \right) \um \Lambda \left( \Psi ; k
     \upl 2 \right)^{^{}} \right] = O_k \left( B^{2 - 3 \cdot 2^{- \left( k
     \upl 1) \right.}} \cdot B^{- \left( k \upl 1 \right) \cdot 2^{- \left( k
     \upl 1) \right.}} \right) \upl o (1) \]
  where $\Delta = B^{1 - 2^{- (k + 1)}}$. Dividing by $\Delta^3 / B \cdot
  \left( \Delta / B \right)^k \asymp B^{2 - 3 \cdot 2^{- (k + 1)}} \cdot B^{-
  k \cdot 2^{- (k + 1)}}$ on both sides, we conclude that $\lambda_f (x ; k +
  2) - \Lambda (\Psi ; k + 2) = O_k (B^{- 2^{- (k + 1)}})$ as desired, thus
  finishing the inductive step. Now, we will prove that $\lambda_f (x ; k)
  \longrightarrow \Lambda (\Psi ; k)$ implies
  \begin{equation}
    \label{momentconvergence}
    \mathcal{M}_f (x ; \ell) \text{ } \assign \text{ } \frac{1}{B (f ; x)^{2}}
    \sum_{p \leqslant x} \frac{f (p)^{\ell + 2}}{p} \longrightarrow
    \int_{\mathbb{R}} t^{\ell} \mathd \Psi (t)
  \end{equation}
  for each fixed $\ell \geqslant 0$. This follows almost immediately from the
  recurrence relation for $\lambda_f (x ; k)$ and $\Lambda (\Psi ; k)$. Indeed
  let us prove (\ref{momentconvergence}) by induction on $k \geqslant 0$. The base case $k = 0$
  is obvious, for the left hand side and right hand side of (\ref{momentconvergence}) are both
  equal to $1$. Let us now suppose that (\ref{momentconvergence}) holds for all $\ell < k$. We
  will prove that convergence also holds for $\ell = k$. \ By definition of
  $\lambda_f (x ; k + 1)$ we have
  \begin{equation}
    \label{transform}
    \lambda_f (x ; k \upl 1) = \um \sum_{j = 2}^k \frac{\mathcal{M}_f (x ; j
    \um 1)}{j!} \sum_{\ell_1 + \ldots + \ell_j = k + 1} \lambda_f (x ; \ell_1)
    \ldots \lambda_f (x ; \ell_j) \um \frac{\mathcal{M}_f (x ; k)}{(k + 1) !}
  \end{equation}
  (we single out $j = k + 1$ on the right hand side). By induction hypothesis
  $\mathcal{M}_f (x ; j \um 1) \longrightarrow \int t^{j - 1} \mathd 
  \Psi (t)$
  as $x \rightarrow \infty$, for $j \leqslant k$. Further as we've shown
  earlier $\lambda_f (x ; i) \longrightarrow \Lambda (\Psi ; i)$ for all $i
  \geqslant 0$. Therefore the whole double sum on the right hand side of
  (\ref{transform}) tends to
  \[ - \sum_{j = 2}^k \frac{1}{j!} \int_{\mathbb{R}} t^{j - 1} \mathd \Psi (t)
  \sum_{\ell_1 + \ldots + \ell_j = k + 1} \Lambda (\Psi ; \ell_1)
     \cdot \ldots \cdot \Lambda (\Psi ; \ell_j) \]
  which, by definition of $\Lambda (\Psi ; k)$ is equal to $\Lambda (\Psi ; k
  + 1) + 1 / (k + 1) ! \int_{\mathbb{R}} t^k \mathd \Psi (t)$. But also
  $\lambda_f (x ; k + 1) \longrightarrow \Lambda (\Psi ; k + 1)$ because
  $\lambda_f (x ; i) \longrightarrow \Lambda (\Psi ; i)$ for all $i \geqslant
  0$. Thus the left hand side of (\ref{transform}) tends to $\Lambda (\Psi ; k + 1)$
  while the double sum on the right hand side of (\ref{transform}) tends to $\Lambda
  (\Psi ; k + 1) + 1 / (k + 1) ! \int_{\mathbb{R}} t^k \mathd \Psi (t)$.
  Therefore equation (\ref{transform}) transforms into
  \[ \Lambda (\Psi ; k + 1) = \Lambda (\Psi ; k + 1) + \frac{1}{(k + 1) !}
     \int_{\mathbb{R}} t^k \mathd \Psi (t) - \frac{\mathcal{M}_f (x ;
     k)}{(k + 1) !} + o_{x \rightarrow \infty} (1) \]
  and $\mathcal{M}_f (x ; k) \rightarrow \int_{\mathbb{R}} t^k \mathd \Psi(t)$
  $(x \rightarrow \infty)$ follows. This establishes the induction step
  and thus (\ref{momentconvergence}) for all fixed $\ell \geqslant 0$. Now we use the method of
  moments to prove that
  \[ F (x ; t) \assign \text{ } \frac{1}{B^2 (f ; x)}
     \sum_{\tmscript{\begin{array}{c}
       p \leqslant x\\
       f (p) \leqslant t
     \end{array}}} \frac{f (p)^2}{p} \text{ } \underset{x \rightarrow
     \infty}{\longrightarrow} \text{ } \Psi (t) \]
  holds at all continuity points $t$ of $\Psi (t)$. Let us note that, the
  $k$-th moment of the distribution function $F (x ; t)$, is given by
  $\mathcal{M}_f (x ; k)$, and as we've just shown this converges to the the
  $k$-th moment of $\Psi (t)$. That is
  \[ \int_{\mathbb{R}} t^k \mathd F (x ; t) = \frac{1}{B^2 (f ; x)}  \sum_{p
     \leqslant x} \frac{f (p)^{k + 2}}{p} \text{ } \longrightarrow \text{ }
     \int_{\mathbb{R}} t^k \mathd \Psi (t) \]
  Since $\Psi (a) - \Psi (0) = 1$ for some $a > 0$, the distribution function
  $\Psi$ satisfies the assumption of Lemma \ref{momentlemma}, hence, by Lemma \ref{momentlemma} (the
  method of moments) we have $F (x ; t) \longrightarrow \Psi (t)$ at all
  continuity points $t$ of $\Psi$ as desired. 
\end{proof}

\section{Primes to integers: Proof of Theorem \ref{MainThm2}}

We keep the same notation as in the previous section. Namely we recall that,
\begin{eqnarray*}
  B^2 (f ; x) \assign \sum_{p \leqslant x} \frac{f (p)^2}{p} & \tmop{and} &
  \mathcal{D}_f \left( x ; \Delta \right) \assign \frac{1}{x} \cdot
  \# \left\{ n \leqslant x : \frac{f (n) - \mu (f ; x)}{B (f ; x)} \geqslant
  \Delta \right\}
\end{eqnarray*}
We first need to modify a little some of the known large deviations results
for $\mathcal{D}_f (x ; \Delta)$ and $\mathbb{P}(\mathcal{Z}_{\Psi}
(B^2 (f ; x)) \geqslant \Delta B (f ; x))$.

\subsection{Large deviations for $\mathcal{D}_f (x ; \Delta)$ and
$\mathbb{P}(\mathcal{Z}_{\Psi} (x) \geqslant t)$ revisited}

First we require the result of Maciulis (\cite{Maciulis}, theorem) in a 
``saddle-point'' version.

\begin{lemma}
  \label{MMainLemma1}
  Let $g$ be a strongly additive function such that $0 \leqslant g (p)
  \leqslant O (1)$ and $B(g;x) \rightarrow \infty$. 
  We have uniformly in $1 \leqslant \Delta \leqslant o (B (g
  ; x))$,
  \[ \mathcal{D}_g (x ; \Delta) \sim \exp \left( \sum_{p \leqslant x}
     \frac{e^{\eta g (p)} \um \eta g (p) \um 1}{p} - \eta \sum_{p \leqslant x}
     \frac{g (p) (e^{\eta g (p)} \um 1)}{p} \right) \frac{e^{\Delta^2 /
     2}}{\sqrt{2 \pi}} \int_{\Delta}^{\infty} e^{- t^2 / 2} \mathd t \]
  where $\eta = \eta_g (x ; \Delta)$ is defined as the unique positive
  solution of the equation
  \[ \sum_{p \leqslant x} \frac{g (p) e^{\eta g (p)}}{p} = \mu (g ; x) +
     \Delta B (g ; x) \]
  Furthermore $\eta_g (x ; \Delta) = \Delta / B (g ; x) + O (\Delta^2 / B (g ;
  x)^2)$.
\end{lemma}

\begin{proof}
  Only the last assertion needs to be proved, because it is not stated
  explicitly in Maciulius's paper. Fortunately enough, it's a triviality.
  Indeed, writing $\eta = \eta_g (x ; \Delta)$, we find that
  \[ 0 \leqslant \eta_g (x ; \Delta) \sum_{p \leqslant x} \frac{g (p)^2}{p}
     \leqslant \sum_{p \leqslant x} \frac{g (p) (e^{\eta g (p)} - 1)}{p} =
     \Delta B (g ; x) \]
  Dividing by $B^2 (g ; x)$ on both sides $0 \leqslant \eta_g (x ; \Delta)
  \leqslant \Delta / B (g ; x)$ follows. Now expanding $e^{\eta g (p)} - 1 =
  \eta g (p) + O (\eta^2 g (p)^2)$ and noting that $O (\eta^2 g (p)^2) = O
  (\eta^2 g (p))$ because $g (p) = O (1)$, we find that
  \[ \Delta B (g ; x) = \sum_{p \leqslant x} \frac{g (p) (e^{\eta g (p)} -
     1)}{p} = \eta \sum_{p \leqslant x} \frac{g (p)^2}{p} + O \left( \eta^2
     \sum_{p \leqslant x} \frac{g (p)^2}{p} \right) \]
  Again dividing by $B^2 (g ; x)$ on both sides, and using the bound $\eta = O
  (\Delta / B (g ; x))$ the claim follows.
\end{proof}

Adapting Hwang's \cite{Hwang} result we prove the following.

\begin{lemma}
  \label{MMainLemma2}
  Let $\Psi$ be a distribution function. Suppose that there is an $\alpha > 0$
  such that $\Psi (\alpha) - \Psi (0) = 1$. Let $B^2 = B^2 (x) \rightarrow
  \infty$ be some function tending to infinity. Then, uniformly in $1
  \leqslant \Delta \leqslant o (B (x))$ the quantity $\mathbb{P} \left(
  \mathcal{Z}_{\Psi} (B^2) \geqslant \Delta B \right)$ is asymptotic to
  \[ \exp \left( B^2 \int_{\mathbb{R}} \frac{e^{\rho u} - \rho u - 1}{u^2}
     \mathd \Psi (u) - B^2 \cdot \rho \int_{\mathbb{R}} \frac{e^{\rho u} -
     1}{u} \mathd \Psi (u) \right) \cdot \frac{e^{\Delta^2 / 2}}{\sqrt[]{2
     \pi}} \int_{\Delta}^{\infty} e^{- t^2 / 2} \tmop{dt} \]
  where $\rho = \rho_{\Psi} (B (x) ; \Delta)$ is defined implicitly, as the
  unique positive solution to
  \[ B^2 (x) \int_{\mathbb{R}} \frac{e^{\rho u} - 1}{u} \mathd \Psi (u) =
     \Delta \cdot B (x) \]
\end{lemma}

\begin{proof}
  We keep the same notation as in lemma \ref{MainLemma2}. While proving lemma \ref{MainLemma2} we
  established the following useful relationship (see (\ref{uinverse}))
  \[ \sum_{k \geqslant 0} \Lambda (\Psi ; k + 2) \xi^{k + 2} = (u')^{- 1}
     (\xi) - \xi \text{ where } u (z) = \int_{\mathbb{R}} \frac{e^{zt} - zt -
     1}{t^2} \mathd \Psi (t) \]
  Here $(u')^{- 1}$ denotes the inverse function of $u'$. \ Integrating the
  above gives
  \[ \sum_{k \geqslant 0} \frac{\Lambda (\Psi ; k + 2)}{k + 3} \cdot \xi^{k +
     3} = - \frac{\xi^2}{2} + \xi \cdot (u')^{- 1} (\xi) - u ((u')^{- 1}
     (\xi)) \]
  Now choose $\xi = \Delta / B$, then by definition $\rho = \rho_{\Psi} (B (x)
  ; \Delta) = (u')^{- 1} (\xi)$. Thus the above formula becomes
  \[ \sum_{k \geqslant 0} \frac{\Lambda (\Psi ; k + 2)}{k + 3} \cdot (\Delta /
     B)^{k + 3} = - \frac{(\Delta / B)^2}{2} + \frac{\Delta}{B} \cdot \rho -
     \int_{\mathbb{R}} \frac{e^{\rho t} - \rho t - 1}{t^2} \mathd \Psi (t) \]
  Also, note that by definition $\Delta / B = \int_{\mathbb{R}} (e^{\rho t} -
  1) / t \cdot \mathd \Psi (t)$. Using the above formula (in which we replace
  $(\Delta / B) \cdot \rho$ by $\rho \int_{\mathbb{R}} (e^{\rho t} - 1) / t
  \cdot \mathd \Psi (t)$) and lemma \ref{MainLemma2}
  \begin{eqnarray*}
    &  & \mathbb{P}(\mathcal{Z}_{\Psi} (B^2) \geqslant \Delta B) \sim \exp
    \left( - B^2 \sum_{k \geqslant 0} \frac{\Lambda (\Psi ; k + 2)}{k + 3}
    \cdot \left( \Delta / B \right)^{k + 3} \right) \int_{\Delta}^{\infty}
    e^{- u^2 / 2} \cdot \frac{\mathd u}{\sqrt{2 \pi}}\\
    & = & \exp \left( B^2 \int_{\mathbb{R}} \frac{e^{\rho t} - \rho t -
    1}{t^2} \mathd \Psi (t) - B^2 \cdot \rho \int_{\mathbb{R}} \frac{e^{\rho
    t} - 1}{t} \mathd \Psi (t) \right) \cdot \frac{e^{\Delta^2 / 2}}{\sqrt[]{2
    \pi}} \int_{\Delta}^{\infty} e^{- u^2 / 2} \mathd u
  \end{eqnarray*}
  This is the claim.
\end{proof}

Before we prove Theorem \ref{MainThm2}, we need to show that the parameters $\eta_g (x ;
\Delta)$ and $\rho_{\Psi} (B (f ; x) ; \Delta)$ (as defined respectively in
Lemma \ref{MMainLemma1} and Lemma \ref{MMainLemma2}) are ``close'' when the distribution of the $g (p)$'s
resembles $\Psi (t)$. The ``closeness'' assertion is made precise in the next
lemma.

\begin{lemma}
  \label{MMainLemma3}
  Let $\Psi (\cdot)$ be a distribution function. Let $f$ be a positive
  strongly additive function. Suppose that $0 \leqslant f (p) \leqslant O (1)$
  for all primes $p$. Let
  \[ K_f (x ; t) \assign \frac{1}{B^2 (f ; x)}
     \sum_{\tmscript{\begin{array}{c}
       p \leqslant x\\
       f (p) \leqslant t
     \end{array}}} \frac{f (p)^2}{p} \]
  If $K_f (x ; t) - \Psi (t) \ll 1 / B^2 (f ; x)$ uniformly in $t \in
  \mathbb{R}$, then
  \[ \rho_{\Psi} (B (f ; x) ; \Delta) - \eta_f (x ; \Delta) = o (1 / B^2 (f ;
     x)) \]
  uniformly in $1 \leqslant \Delta \leqslant o (B (f ; x))$. The symbols
  $\rho_{\Psi} (B (x) ; \Delta)$ and $\eta_f (x ; \Delta)$ are defined in
  lemma \ref{MMainLemma2} and lemma \ref{MMainLemma1} respectively. 
\end{lemma}

\begin{proof}
  Let $\eta = \eta_f (x ; \Delta)$. Recall that by lemma \ref{MMainLemma1}, 
  $\eta = o (1)$ in the range $1 \leqslant \Delta \leqslant o(B(f;x))$. 
  This will
  justify the numerous Taylor expansions involving the parameter $\eta$. With
  $K_f (x ; t)$ defined as in the statement of the lemma, we have
  \begin{eqnarray}
    \label{approx1}
    \sum_{p \leqslant x} \frac{f (p) e^{\eta f (p)}}{p} - \mu (f ; x) & = &
    \sum_{p \leqslant x} \frac{f (p) (e^{\eta f (p)} - 1)}{p}_{} \nonumber\\
    & = & B^2 (f ; x) \int_{\mathbb{R}} \frac{e^{\eta t} - 1}{t} \mathd K_f
    (x ; t) 
  \end{eqnarray}
  Let $M > 0$ be a real number such that $0 \leqslant f (p) \leqslant M$ for
  all $p$. Since the $f (p)$ are bounded, for each $x > 0$ the distribution
  function $K_f (x ; t)$ is supported on $[0 ; M]$. Furthermore since $K_f (x
  ; t) \rightarrow \Psi (t)$ the distribution function $\Psi (t)$ is supported
  on exactly the same interval. From these considerations, it follows that
  \begin{eqnarray}
    \label{approx2}
    &  & \int_{\mathbb{R}} \frac{e^{\eta t} - 1}{t} \mathd K_f (x ; t)
    \text{ } = \text{ } \int_0^M \frac{e^{\eta t} - 1}{t} \mathd K_f (x ; t)
    \nonumber\\
    & = & \int_0^M \frac{e^{\eta t} - 1}{t} \mathd \Psi (t) + \int_0^M \left(
    K_f (x ; t) - \Psi (t)^{^{^{}}} \right) \cdot \left[ \frac{e^{\eta t} -
    \eta t \cdot e^{\eta t} - 1}{t^2} \right] \mathd t 
  \end{eqnarray}
  By a simple Taylor expansion $e^{\eta t} - \eta t \cdot e^{\eta t} - 1 = O
  (\eta^2 t^2)$. Therefore the integral on the right hand side is bounded by
  $O (\eta^2 / B^2 (f ; x))$. We conclude from (\ref{approx1}) and (\ref{approx2}) that
  \begin{equation}
    \label{close1}
    \sum_{p \leqslant x} \frac{f (p) e^{\eta f (p)}}{p} - \mu (f ; x) = B^2 (f
    ; x) \int_{\mathbb{R}} \frac{e^{\eta t} - 1}{t} \mathd \Psi (t) + O
    \left( \eta^2 \right)
  \end{equation}
  By definition of $\rho_{\Psi}$ and $\eta_f$,
  \begin{equation}
    \label{close2}
    B^2 (f ; x) \int_{\mathbb{R}} \frac{e^{\rho_{\Psi} (B ; \Delta) t} -
    1}{t} \mathd \Psi (t) = \Delta B (f ; x) = \sum_{p \leqslant x} \frac{f
    (p) e^{\eta_f (x ; \Delta) f (p)}}{p} - \mu (f ; x)
  \end{equation}
  From (\ref{close1}) and (\ref{close2}) it follows that
  \begin{equation}
    \label{final1}
    B^2 (f ; x) \int_{\mathbb{R}} \frac{e^{\rho_{\Psi} (B ; \Delta) t} -
    e^{\eta_f (x ; \Delta) t}}{t} \mathd \Psi (t) = O \left( \eta^2_f (x ;
    \Delta) \right)
  \end{equation}
  Since $\Psi (t)$ is supported on $[0 ; M]$ we can restrict the above
  integral to $[0 ; M]$. By lemma \ref{MMainLemma1}, we have $\eta_f (x ; \Delta) \sim
  \Delta / B (f ; x) = o (1)$ in the range $\Delta \leqslant o (B (f ; x))$.
  Also $0 \leqslant \rho_{\Psi} (x ; \Delta) \leqslant \int_{[0 ; M]}
  (e^{\rho_{\Psi} (x ; \Delta) t} - 1) / t \cdot \mathd \Psi (t) = \Delta / B
  (f ; x) = o (1)$ for $\Delta$ in the same range. Write $\rho \assign
  \rho_{\Psi} (B ; \Delta)$ and $\eta \assign \eta_f (x ; \Delta)$. For $0
  \leqslant t \leqslant M$, we have
  \begin{eqnarray*}
    (1 / t) \left( e^{\rho t} - e^{\eta t} \right) & = & (1 / t) e^{\rho t}
    \cdot (1 - e^{(\eta - \rho) t})\\
    & = & e^{\rho t} \cdot (\eta - \rho) + O ((\eta - \rho)^2) \text{ }
    \asymp \text{ } \eta - \rho
  \end{eqnarray*}
  because $\rho = o (1)$, $\eta = o (1)$. Inserting this estimate into
  (\ref{final1}) we get $\eta - \rho = O (\eta^2 / B^2 (f ; x)) = o (1 / B^2 (f ;
  x))$ since $\eta^2 = o (1)$. The lemma is proved. 
\end{proof}

\subsection{Proof of Theorem \ref{MainThm2}}

\begin{proof}[Proof of Theorem \ref{MainThm2}]
  Let notation be as in Lemma \ref{MMainLemma2}. Let $\eta = \eta_f (x ; \Delta)$ be the
  parameter from lemma \ref{MMainLemma1}. Proceeding as in the proof of the previous lemma,
  we get
  \begin{eqnarray}
    \label{closer1}
    \sum_{p \leqslant x} \frac{e^{\eta f (p)} - \eta f (p) - 1}{p} & = & B^2
    (f ; x) \int_{\mathbb{R}} \frac{e^{\eta u} - \eta u - 1}{u^2} \mathd \Psi
    (u) + o (1) \\
    \label{closer2}
    \eta \sum_{p \leqslant x} \frac{f (p) \cdot (e^{\eta f (p)} - 1)}{p} & = &
    B^2 (f ; x) \eta \int_{\mathbb{R}} \frac{e^{\eta u} - 1}{u} \mathd \Psi
    (u) + o (1) 
  \end{eqnarray}
  throughout the range $1 \leqslant \Delta \leqslant o (B (f ; x))$. Let $\rho
  \assign \rho_{\Psi} (B (f ; x) ; \Delta)$ denote the parameter from Lemma
  \ref{MMainLemma2}. The functions on the right of $(10.6)$ and $(10.7)$ are analytic.
  Therefore, by lemma \ref{MMainLemma3},
  \begin{eqnarray}
    \label{result1}
    B^2 (f ; x) \int_{\mathbb{R}} \frac{e^{\eta u} - \eta u - 1}{u^2} \mathd
    \Psi (u) & = & B^2 (f ; x) \int_{\mathbb{R}} \frac{e^{\rho u} - \rho u -
    1}{u^2} \mathd \Psi (u) \upl o (1) \\
    \label{result2}
    B^2 (f ; x) \eta \int_{\mathbb{R}} \frac{e^{\eta u} - 1}{u} \mathd \Psi
    (u) & = & B^2 (f ; x) \rho \int_{\mathbb{R}} \frac{e^{\rho u} - 1}{u} 
    \mathd \Psi (u) + o (1) 
  \end{eqnarray}
  uniformly in $1 \leqslant \Delta \leqslant o (B (f ; x))$. On combining
  (\ref{closer1}) with (\ref{result1}) and (\ref{closer2}) with (\ref{result2}) we obtain
  \begin{eqnarray*}
    &  & \sum_{p \leqslant x} \frac{e^{\eta f (p)} - \eta f (p) - 1}{p} -
    \eta \sum_{p \leqslant x} \frac{f (p) (e^{\eta f (p)} - 1)}{p}\\
    & = & B^2 (f ; x) \int_{\mathbb{R}} \frac{e^{\rho u} - \rho u - 1}{u^2}
    \mathd \Psi (u) - B^2 (f ; x) \rho \int_{\mathbb{R}} \frac{e^{\rho u} -
    1}{u} \mathd \Psi (u) + o(1)
  \end{eqnarray*}
  By lemma \ref{MMainLemma1}, lemma \ref{MMainLemma2} and the above equation, we get
  \[ \mathcal{D}_f (x ; \Delta) \sim \mathbb{P} \left(
     \mathcal{Z}_{\Psi} \left( B^2 (f ; x) \right) \geqslant \Delta B (f ; x)
     \right) \]
  uniformly in $1 \leqslant \Delta \leqslant o (B (f ; x))$ as desired. 
\end{proof}

{\noindent}\tmtextbf{Acknowledgements.} This is part of author's undergraduate thesis, written under the direction of
Andrew Granville. The author would like to thank first and foremost Andrew Granville. There is too much to thank for, so 
it is simpler to note that this project would not surface without his constant support. Also, the author would like
to thank Philippe Sosoe for 
proof-reading a substantial part of the old manuscript of this paper. 
{\hspace*{\fill}}{\medskip}

\bibliography{newpaper2}
\bibliographystyle{plain}

\end{document}